\documentclass[10pt, reqno]{amsart}
\usepackage{amsmath, amsthm, amscd, amsfonts, amssymb, graphicx, color}
\usepackage{paralist}
\usepackage[bookmarksnumbered, colorlinks, plainpages]{hyperref}
\newtheorem{theorem}{Theorem}[section]
\newtheorem{lemma}[theorem]{Lemma}

\newtheorem{corollary}[theorem]{Corollary}
\theoremstyle{definition}
\newtheorem{definition}[theorem]{Definition}

\theoremstyle{remark}
\newtheorem{remark}[theorem]{Remark}
\numberwithin{equation}{section}
\usepackage{lineno}
\hypersetup{pdfstartview={XYZ null null 1.00}}

\newcommand{\V}{{\mathcal V}} 
\newcommand{\U}{{\mathcal U}} 
\newcommand{\kx}{{\mathcal K(X)}}
\newcommand{\kz}{{\mathcal K(Z)}}

\newcommand{\cpx}{{C_{p}(X)}} 
\newcommand{\qpx}{{Q_{p}(X)}} 
\newcommand{\qcx}{{Q_{C}(X)}} 
\newcommand{\R}{{\mathbb R}} 
\newcommand{\N}{{\mathbb N}}
\newcommand{\alephnull}{{\aleph_{0}}} 

\newcommand{\cx}{{C(X)}} 

\newcommand{\qx}{{Q(X)}} 

\begin{document}
\setcounter{page}{1}
\title[Cardinal Properties of the Space $\qcx$]{Cardinal Properties of the Space of Quasicontinuous Functions under Topology of Uniform Convergence on Compact Subsets}
\author[Chander Mohan Bishnoi*]{Chander Mohan Bishnoi}

\address{Chander Mohan Bishnoi\\
Department of Mathematics\\
Lovely Professional University, Punjab, India}
\email{chandermohan.cm.b@gmail.com}
\thanks{}
\author{Sanjay Mishra}
\address{Sanjay Mishra\\
Department of Mathematics\\
Amity School of Applied Sciences  \\
Amity University, Lucknow Campus, Uttar Pradesh, India}
\email{drsmishraresearch.com}
\date{Received: xxxxxx; Revised: yyyyyy; Accepted: zzzzzz.
\newline \indent $^{*}$Corresponding author}
\begin{abstract}
In this paper, we investigate various cardinal properties of the space $Q_{C}X$ of all real-valued quasicontinuous functions on the topological space $X$, under the topology of uniform convergence on compact subsets. It begins by examining the relationship between tightness and other properties in the context of the space $X$, highlighting results such as the alignment of tightness $Q_{C}X$ with the compact Lindelöf number of $X$ under Hausdorff conditions and the countable tightness of $Q_{C}X$ when $X$ is second countable. Further investigations reveal conditions for the tightness of $Q_{C}X$ relative to $k$-covers of $X$, as well as connections between density tightness, fan tightness, and other properties in Hausdorff spaces. Additionally, we discuss the implications of the Frechet-Urysohn property $Q_{C}X$ for open $k$-covers in Hausdorff spaces. We explore relationships between $Q_{C}X$'s tightness, the Frechet-Urysohn property, and the $\sigma$-compactness of locally compact Hausdorff spaces $X$. Furthermore, we examine the $k_{f}$-covering property and the existence of $k$-covers in the context of Whyburn spaces.
\end{abstract} \maketitle

\section{Introduction}\label{s:Introduction}
Before we go deeper into the content of this article, it is necessary to first look at the essential symbols and explanations that will be adopted. The notation to be used consists of $X$, $Y$, and $Z$, which should be considered as topological spaces unless clearly stated differently. The set $\qx$ represents collections of quasicontinuous real-valued functions on the space $X$, the  $\qcx$ is the space of the set $\qx$ endowed with the topology of uniform convergence on a compact set,  the set of natural numbers, real numbers, and rational numbers are $\N, \R$ and  $\mathbb{Q}$, respectively. $|A|$ is the cardinality of $A$, $\alephnull$ is a countable cardinal number (first infinite cardinal number), and $\kappa$, $\lambda$  represents any cardinal number.  Also, $\kz$ stands for the collection of all compact subsets of space $Z$, a closure of set $A$  is $\overline{A}$ and the difference of two set $X$, $A$ is $X/A$.
Baire derived the condition of quasicontinuous functions in his article \cite{Baire1899} while studying the points of continuity for separately continuous functions from $\R^{2}$ to $\R$. Kempisty extensively investigated the quasicontinuity suggested in his paper \cite{Kempisty1932} for real functions of several variables. When we examine deeper, we notice that the researchers were interested in this study for various reasons.  The two most important of these are: The first reason is it provides a relatively good connection to continuity while being a more general concept. This relationship allows mathematicians to extend concepts and theorems from continuous functions to a broader class of functions. The second reason is that it has deep connections to mathematical analysis and topology. Quasicontinuous functions were significant in the studies of the characterization of minimal cusco, minimal usco maps, topological groups, and the CHART group, which is the basic object for studying topological dynamics. More information on using quasicontinuous functions can be found in [\cite{Hola2021}, Chapter 2] and a recent study of quasicontinuous functions found in \cite{Bishnoi2023}.

 Let us take a look at how the study of topological function space began, with functions considered as quasicontinuous functions. Hola and Holy investigated the features of the space of quasicontinuous functions under different topologies in the literature. \cite{Hola2016, Hola2017, Hola2018, Hola2020, Kumar2022}. Similarly, the space of continuous functions under different topologies is studied in \cite{Aaliya2023, Arkhangelskii992, Tkachuk2010, Tkachuk2014, Tkachuk2015, Mishra2023, Aaliya2024, Aaliya2023some}. Recently, in 2021 \cite{hola2021quasicontinuouscompact} Hola and Holy studied the space $\qcx$ in which topology is defined by using uniformity. They studied some cardinal functions for the space $\qcx$. Furthermore, in continuation, we investigated more interesting cardinal functions such as density, tightness, and fan tightness of $\qcx$. Moreover, we derived some equivalent conditions for the Frechet-Uryshon property of $\qcx$.

 In section \eqref{s:Pre-requisites}, we revisit some important concepts and conclusions to help understand our main work.
In section \eqref{s:Various types of tightness of qcx}, we present several types of results referring to the relationship between various topological properties of space. Firstly, under the condition of Hausdorffness of the space $X$, the tightness of the $\qcx$ coincides with the compact Lindelöf number of $X$, and under the condition of second countability of $X$, $\qcx$ has countable tightness. Additionally, the necessary and sufficient condition for the tightness of $\qcx$ to be less than $\kappa$ in terms of $k$-covers of $X$ is determined. Furthermore, under the condition of Hausdorffness of $X$, the density tightness and tightness of $\qcx$ coincide, and necessary and sufficient conditions for the countable infinite of the fan tightness of $\qcx$ under the Hausdorffness condition of $X$ are found. Moreover, conditions under which $\qcx$ has countable strong fan tightness are established. Finally, an inequality between the density of $\qcx$ and the cofinality of $X$ is established when the space $X$ is regular. 

In section \eqref{s:Frechet-Urysohn Property of qcx}, we demonstrate that for a Hausdorff space, an open $k$-cover possesses a countable $k$-subcover if it $\qcx$  is a Frechet-Urysohn space. Additionally, we establish the relationship between the $\sigma$-compactness of the locally compact Hausdorff space $X$, the tightness of $\qcx$, and the Frechet-Urysohn property of $\qcx$. The $k_{f}$-covering property of $X$ is explored under the condition that  $\qcx$ is a Frechet-Urysohn space and studies the existence of a $k$-cover for $X$ in the context of  $\qcx$ as a Whyburn space. These investigations provide valuable insights into the topological structures of spaces and elucidating connections between various fundamental properties.

\section{Preliminaries}\label{s:Pre-requisites}
According to Neubrunn in  \cite{Neubrunn1988}, ``a function $f \colon X \to Y$ is known as quasicontinuous at $x \in X$ if there exists a non-empty open subset $C$ of $A$ such that $f(C)\subset B$, where  $A$ and $B$ are any open sets in $X$ and $Y$ respectively with $x \in A$, $f(x) \in B$". If for all $x \in X$, $f$ is quasicontinuous, then it is quasicontinuous on $X$. Now, let us recall certain definitions related to our main work.
For any space $Z$, the density  and $k$-cofonality are defined as
\[d(Z) = \alephnull + \min\{|T|\colon \overline{T} = Z\}, \text{and}\]
\[kcof(Z) = \alephnull + \{|\mathcal{C}|\colon \mathcal{C} \, \text{is a cofinal family of} \, \kz\}\]

The tightness and fan tightness of a point $z \in Z$ are defined as:\\
$t(z,Z) = \alephnull + \min\{\kappa \colon \text{for each  $T\subset Z$, if}\,  z \in \overline{T},\text{There exists}\, B\subset T\,\text{such that}\, z \in \overline{B}\,\text{and}\, |B|\leq \kappa\}$
and \\ $ft(z,Z) = \alephnull + \min\{\kappa \colon  \text{for each}\, \{T_{n}\} \text{sequence  of  subsets} \, Z \, \text{and} \, z \in \cap_{n\in\mathbb{N}}\overline{T_{n}}, \text{there exists}\, B_{n} \subset T_{n} \, \text{such that}\, z\in\overline{\cup_{n\in\mathbb{N}}A_{n}}, \text{with}\, |B_{n}|\leq \kappa\}$
 The tightness and fan tightness of space $Z$ are $t(Z) = \sup\limits_{z\in Z}\{t(z,Z)\}$  and $ft(Z) = \sup\limits_{z\in Z}\{ft(z,Z)\}$, respectively. Furthermore, for a space $Z$ if for every $z\in Z$,  $\{T_{n}\}_{n\in \N}$ is a sequence of subsets of $Z$ with $z\in \bigcap_{n\in \N} \overline{T_{n}}$, there exists $z_{n}\in T_{n}$ satisfying $z\in \overline{\{z_{n}\colon n\in \N\}}$, then $Z$ has countable strong fan tightness.

 A space $Z$ is said to be
\begin{enumerate}
\item $\sigma$-compact \cite{engelking1989}, if $Z=\bigcup\limits_{n\in\N}C_{n}$, where $C_{n}$ is a compact subset of $Z$.
\item Hemicompact \cite{engelking1989}, if any collection $\mathcal{C}\subset \kz$ has subset ordering, then exists a countable cofinal subfamily.
\item Frechet-Urysohn \cite{Arkhangelskii992}, if for any $z\in Z$ and $S\subset Z$ with $z\in \overline{S}$, there exists a sequence $\{z_{n}\}_{n\in\N}$ in $S$ converges to $z$.
\item Whyburn \cite{osipov2023} if for any subset $S$ of $Z$ and any $z \in \overline{S}/ S$ there is a set $P \subset S$ such that $\overline{P} = S \cup \{z\}$.
\end{enumerate}
\begin{remark}\cite{engelking1989}
Hemicompactness implies space $\sigma$-compact, but not vice versa.
\end{remark}

\begin{definition}\cite{Okunev2002}
A subset $T$ of the space $Z$ is $\kappa$-dense in $Z$ if $[T]_{\kappa}=Z$, where $[T]_{\kappa}$  denotes the $\kappa$-closure of set $T$ and it is defined as
\begin{equation*}\label{eq:k-closure}
[T]_{\kappa} = \bigcup\{|\overline{Y}|\colon Y \, \text{is subset of} \, T \, \text{and} \, |Y| \leq \kappa \}
\end{equation*}
\end{definition}

\begin{definition}\cite{Okunev2002}
The density tightness of a space $Z$ is denoted by $dt(Z)$ and defined as:
\begin{equation*}\label{eq:Density tightness of X}
dt(Z) = \min\{\kappa \colon \text{Each dense subset is $\kappa$-dense in} \, Z \}
\end{equation*}
\end{definition}
\begin{proof}[\cite{Okunev2002}, Prop. 2.2]\label{t:Density tightness and tightness relation}
For any space $Z$, then $dt(Z)\leq t(Z)$.
\end{proof}
\begin{lemma}\cite{Hola2020}\label{l:condition quasi}
Let $Z$ be a Hausdorff space, and let $g \colon Z \to \mathbb{R}$ be a function such that, for each $z \in Z$, there exists an open set $U$ in $Z$ containing $z$ where $g(y) = g(z), \, \forall \, y \in U$. Then $g$ is quasicontinuous.
\end{lemma}

\begin{lemma}\cite{Kumar2022}\label{l:quasicontinuous function separating point and closed set}
If $Z$ is a regular space and $Y$ is any space. For any  $z\in Z$, a non-empty closed subset $C$ of $Z$ not containing $z$ and $c_{1}, c_{2}\in C$, there exists a quasicontinuous function $f\colon Z \to Y$ satisfying  $f(x) = c_{1}$ and $f(C) = \{c_{2}\}$.
\end{lemma}

The topology of uniform convergence on  compact set on the set $\qx$ in  
 \cite{hola2021quasicontinuouscompact}. This topology is induced by the uniformity as `` the basis elements for the uniformity is $W(K,\epsilon) = \{(f,g)\colon f,g\in \qx\,\text{and}\, |f(x)-g(x)| < \epsilon, \, \forall \, x \in K\}$, where $K \in \kx$ is set of all compact subsets of $X$ and $\epsilon > 0$. The basis neighborhood of $f \in \qx$ for the topology on $\qx$ is $W(f, K, \epsilon) = \{g \in \qx \colon\, \text{and}\, |f(x)-g(x)|< \epsilon, \, \forall\, x \in K\}$".
\section{Tightness properties of the space $\qcx$}\label{s:Various types of tightness of qcx}
In this section,  we studied how the cardinality of covers for compact subsets of space $X$ is used to analyze the tightness of $\qcx$, we define that cover as $k$-covers,  and its cardinality is known as the compact Lindelöf number. Which measures the size of the smallest open cover needed to cover the entire collection of compact subsets of a topological space. It provides information about the ``coverage efficiency” of open covers in a topological space.
A smaller compact  Lindelöf number indicates that the space can be covered by a relatively small number of open sets, while a larger compact  Lindelöf number suggests that more open sets are needed for a complete cover. Also, the tightness of a topological space measures the degree to which the closed sets can be separated from each other by open sets. It provides insights into the ``closeness” of closed sets in the space and helps classify and compare different topological spaces based on their separation properties. Further, we characterize the fan tightness of $\qcx$ using the k-covering properties of $X$. Furthermore, we examine the condition of coincidence of density, tightness, and the tightness of $\qcx$ and how the dominating density for $\qcx$ is the $k$-cofinality of $X$.

\begin{definition}\label{d:k-cover}
A collection $\V$ containing subsets of a space $X$ is said to be a $k$-cover of $X$ if, for every  $C \in \kx$,  there exists $V \in \V$ which contains $C$. If every element of $\V$ is open, then it is called an open $k$-cover.
\end{definition}

\begin{definition}\label{d:compact lindelof number}
For any space $X$, the compact Lindelöf number is denoted and defined by
\begin{equation*}\label{eq:Lindelöf number of X}
kL(X) = \alephnull + \inf\{\lambda\colon \text{For each  open $k$-cover $\U$ of}\, X, \, \text{There exists}\, \text{$k$-subcover} \, \mathcal{W}\, \text{with} \,|\mathcal{W}|\leq \lambda\}
\end{equation*}
\end{definition}
\begin{theorem}\label{t:Compact lindeloff no. condition on k cover}
Let $X$ be any space, then $kL(X)\leq \lambda$ if and only if for each open $k$-cover $\V$ of  $X$, there exists a $k$-subcover ${\V}'$ of $X$ with $|{\V}'|\leq \lambda$.
\end{theorem}
\begin{proof}
Consider that  for every open $k$-cover $\V$ of $X$, there  exists a $k$-subcover ${\V}'$ of $X$ having cardinality less then or equal to $ \lambda$. Then, simply by the definition of compact Lindelof number, $kL(X)\leq \lambda$. Conversely, if $kL(X)\leq \lambda$ and $\V$ is any  open  $k$-cover space $X$. Thus, by the definition of compact Lindelof number, there is a $k$-subcover ${\V}'$ of $X$ having cardinality $\leq \lambda$.
\end{proof}

\begin{theorem}\label{t:Compact lindelof no. of x and tightnessof qcx}
For any  Hausdorffspace$X$, then $t(\qcx)=kL(X)$.
\end{theorem}
\begin{proof}
Suppose $t(\qcx)=\lambda$. Now, let $\U$ be any open $k$-cover of $X$. Thus, by Definition \eqref{d:k-cover}, for every $K\in \kx$, we have a $U_{K}\in\U$ such that $K\subset U_{K}$. Let $f_{K}\colon X \to \R$ be a function defined by $f_{K}(\overline U_{K})=\{0\}$ and $f_{K}(X/\overline U_{K})\subset \{1\}$, then by Lemma \eqref{l:condition quasi}, $f_{K}$ is quasicontinuous for each $K\in \kx$. Let $F = \{f_{K}\colon K\in \kx\}$ be a collection of quasicontinuous functions on $X$. Take $f_{0}$ to be a zero function on $X$, then $f_{0}\in \overline{F}$. Since $t(\qcx)=\lambda$, by the definition of tightness, there is  a subset $F'$ of $F$ with property $f_{0}\in \overline{F'}$ and $|F'|\leq \lambda$. Let us consider a subfamily  $\mathcal{W} = \{U_{K}\colon f_{K}\in F'\}$ of $\U$. Now we claim $W$ is a $k$-cover of $X$. Let $K\in \kx$ and $W(f_{0},K,1)$ as a neighborhood of $f_{0}$, then there is a $K'\in \kx$ such that $f_{K'} \in F'\cap W(f_{0}, K, 1)$. Thus, for $x \in K$, we have
\[f_{K'}(x) < 1, \, \text{if} \, x \in \overline{U_{K'}}, \, f_{K'}(x)=1,\,\text{otherwise}\,\]
This implies that $K\subset U_{K'}$, Therefore, $\mathcal{W}$  is a $k$-cover of $X$. Hence, $kL(X)\leq t(\qcx)$.

 Next, to prove $t(\qcx)\leq kL(X)$.  Suppose  $kL(X)= \lambda$. Let $\U$ be any open $k$-cover. By the definition of $kL$-number,   there is a subset ${\U}'$ of $\U$, which is $k$-subcover of $X$ with  $|{\U}'|\leq \lambda$.
 For any $f\in \qcx$, let us define a function $f_{U_{K}}\colon X \to \R$ with $f_{U_{K}}(x) = f(x)$ if $x \in  U_{K}$  and $f_{U_{K}}(x)=1$ if  $x\in X/U_{K}$. So, for each $U_{K} \in \U$,  by Lemma \eqref{l:condition quasi}, $f_{U_{K}}$ is quasicontinuous. Let $F = \{f_{U_{K}}\colon U_{K}\in \U \}$. From the construction of $f_{U_{K}}$, it's clear that, $f(x)\in \overline{F}$ and we have a subset $F' =\{f_{U_{K}} \colon U_{K}\in {\U}'\}$ of $F$ whose closure contains $f$. Then, by definition of tightness, $t(f, \qcx) = |F'|\leq \lambda$. Since $f$ be any function quasicontinuous function, thus $t(\qcx)\leq kL(X)$.
\end{proof}

\begin{corollary}
For a second countable $X$, then space $\qcx$ has countable tightness.
\end{corollary}
\begin{theorem}\label{t:Tightness and k cover}
For Hausdorff  space $X$, $t(\qcx)\leq \lambda$ if and only if for every open $k$-cover $\U$ of space $X$, there is a $k$-subcover ${\U}'$ of $X$ with $|{\U}'|\leq \lambda$.
\end{theorem}
\begin{proof}
Simply by using Theorems \eqref{t:Compact lindeloff no. condition on k cover} and  \eqref{t:Compact lindelof no. of x and tightnessof qcx}.
\end{proof}

\begin{theorem}
If $X$ be Hausdorff space, then $dt(\qcx) = t(\qcx)$.
\end{theorem}
\begin{proof}
Firstly, we claim that $t(\mathcal{U}_X) \leq dt(\mathcal{U}_X)$. To establish this, by Theorem \eqref{t:Tightness and k cover}, it suffices to demonstrate that for every open $k$-cover $\mathcal{U}$ of  $X$, there exists a $k$-subcover $\mathcal{U}'$ of $X$ with  $|\mathcal{U}'| \leq \lambda$. Take
\[A =\{f\in \qcx \colon \, \text{for some}\, U\in \U \, \text{such that}\, f(X/U)\subset \{0\} \}\]
Since $\U$ is the open $k$-cover of $X$, then we have a function
\[h(x) = g(x)\, \text{if} \, x \in U\, \text{and} \, h(X/U)\subset \{0\}\]
For any open set $W(f, K, \epsilon)$ in $\qcx$, then there is a $U \in \U$ such that $K \subset U$. Take $g \in W(f, K, \epsilon)$, then we have $h\in \qcx$ such that  $h\in  W(f, K, \epsilon) \cap A$. Therefore, $A$ is dense in $\qcx$. Let $f_{0}$ be a zero function, then by definition of density tightness, there is a subset $B$ of $A$ having cardinality less than or equal to $ \lambda$ and $f_{0}\in \overline{B}$. Take ${\U}' = \{U_{f} \colon f \in B\}$, clearly $|{\U}'| \leq \lambda$. Let any $K \in \kx$ and $W(f_{0}, K, 1)$ is a neighborhood of $f_{0}$, then there exists a function $g \in B$ such that $g \in W(f_{0}, K,1)$, then $K\subset U_{g}$. Thus, ${\U}'$ is $k$-cover of $X$.

Next to prove $dt(\qcx)\leq t(\qcx)$, it follows from Theorem \eqref{t:Density tightness and tightness relation}, we have $dt(\qcx)\leq t(\qcx)$.
\end{proof}
The sum of real-valued continuous functions $f$ and quasicontinuous function $g$ on $X$ is $f+g$ defined as $(f + g)(x) = f(x) + g(x)$  being quasicontinuous. The mapping $h_f \colon \qpx \to \qpx$ defined as $h_{f}(g) = f + g$ is continuous for every  $f \in \cpx$ see in [\cite{Kumar2022}, Proposition 5.4]. Consequently, we have the following results for $\qcx$.

\begin{lemma}\label{l:Homogenity}
For any real-valued continuous and quasicontinuous maps $f$ and $g$ on $X$, respectively, a mapping $h_{f} \colon \qcx \to \qcx$ such that $h_{f}(g) = f + g$ forms a homeomorphism.
\end{lemma}

\begin{theorem}
If $X$ is a Hausdorff space, then $ft(\qcx)=\alephnull$ if and only if for each sequence $\{{\U}_{n}\}_{n\in\N}$ of open $k$-covers of $X$, there is a finite subset ${{\U}_{n}}^{'}$ of ${\U}_{n}$ with $\bigcup_{n\in \N}{{\U}_{n}}^{'}$  becomes an open $k$-cover of $X$.
\end{theorem}

\begin{proof}
 Consider $\{{\U}_{n}\}_{n\in\N}$ is  a sequence of open $k$-covers of $X$. Define
\[A_{n}=\{f\in\qcx \colon \,\text{For some}\, U\in {\U}_{n}\, \text{such that}\, f(X/U)\subset \{0\} \},\]
for all $n\in \N$. Since ${\U}_{n}$ is an open $k$-cover of  $X$ and  any open subset $W(f, K, \epsilon)$ of  $\qcx$, there exists a set $U$ in ${\U}_{n}$ such that $K \subset U$. Take $g\in W(f,K,\epsilon)$. Then we define $h\in \qcx$ as
\[h(x)=g(x)\, \, \text{if}\, x\in U\, \text{and}\, h(X/U)\subset \{0\}.\]
Then $h\in A_{n}\cap W(f,K,\epsilon)$, therefore, $A_{n}$ is dense in $\qcx$. Now, let's choose $g_{1}\in \qcx$ defined as $g_{1}(x)= 1, \,\forall\,x\in X $. Since each $A_{n}$ is dense, $g_{1}\in \bigcap_{n=1}^{\infty}\overline{A_{n}}$, so there is a finite set $B_{n}$ of $A_{n}$ such that $g_{1}\in \overline{\bigcup_{n=1}^{\infty}B_{n}}$.

Construct $B_{n}=\{f_{(n, j)}\in A_{n}\colon j\leq i(n)\}$, where $i(n)\in \N$ is finite for each $n\in \mathbb{N}$. For each pair $\{n,j\}$ of natural numbers, there exists some $U_{(n,j)}\in \U_{n}$ satisfying $f_{(n,j)}(X/U)\subset \{0\}$. Let $\U'_n=\{U_{(n,j)}\in \U_{n}\colon j\leq i(n)\}$. Next, we have to prove that $\bigcup_{n\in \N}\U'_{n}$ is a $k$-cover of $X$. For any set $K\in \kx$, there is a function $f_{(n,j)}\in W(g_{1},K,1)$ for some $n,j\in\N$ with $j\leq i(n)$. So there exists a set $U_{(n,j)}\in \bigcup_{n\in \N}\U'_{n}$ with $K \subset U_{(n,j)}$. Therefore, the $\bigcup_{n\in \N}\U'_{n}$ is an open $k$-cover of $X$.

Conversely, take $f_{0}$ be a zero function. By Lemma \eqref{l:Homogenity},  $\qcx$ is homogeneous, now it is enough to show that $ft(f_{0},\qcx)=\alephnull$. Let $f_{0}\in \bigcap_{n\in \N} \overline{A_{n}}$, where $A_{n}\subset \qcx$. For every $n\in \N$ and  function $f\in A_{n}$, there exists  some open subset of $X$ such that the image of that set under function $f$ is contained in $(-\frac{1}{n},\frac{1}{n})$, we denote it as $O_{n,f}$.  Construct ${\U}_{n}=\{O_{n,f}\colon f\in A_{n}\}$ for each $K\in \kx$ there exists a function $f\in W(f_{0},K,1/n)\cap A_{n}$ such that $K\subset O_{n,f}$. Now, let's consider two cases.

\begin{enumerate}
\item If $M$ is infinite. For every $W(f_{0},K,\epsilon)$ neighborhood of $f_{0}$ and $\epsilon>0$, there exists some $m\in\N$ such that  $\frac{1}{m}<\epsilon$. Then, from  construction of ${\U}_{m}$, we have some $g_{m}\in A_{m}$ with $g_{m}(X)=(-\frac{1}{m},\frac{1}{m})$ and $g_{m}\in W(f_{0},K,\epsilon)$. Thus, the sequence $\{g_{m}\}_{m\in \N}$ is convergent to $f_{0}$. Here $B_{m}=\{g_{m}\}_{m\in \N}$ satisfies that $f_{0}\in \overline{\bigcup_{m\in \N}  B_{m}}$.
\item If $M$ is finite. Then  there exists $n_{0}\in\N$ such that \[g(X)\neq (-\frac{1}{m},\frac{1}{m})\, \, \text{whenever}\, m\geq n_{0},\,g\in A_{m}.\]
Since $\{\U_{m}\}_{m\geq n_{0}}$ is a sequence of open $k$-cover, thus we have a finite subset $\U'_{m}$ of $\U_{m}$ such that $\bigcup_{m\geq n_{0}}\U'_{m}$ is an open $k$-cover of $X$. Let it be denoted as
\[{{\U}^{'}}_{m}=\{U_{(m,j)}\in {\U}_{m}\colon j\leq i(m)\}.\]
Then, from construction of ${\U}_{m}$, there exists some $f
_{(m,j)}\in A_{m}$ with \\$f_{(m,j)}(U_{(m,j)})\subset (-\frac{1}{m},\frac{1}{m})$. Now we prove $f_{0}\in \overline{\{f_{(m,j)}\colon j\leq i(m)\}}_{m\in \N}$.
For any neighborhood $W(f_{0},K,\epsilon)$ of $f_{0}$, let
\[H=\{\{m,j\}\colon m\geq n_{0},\, j\leq i(m)\,\text{and}\, K\subset U_{(m,j)})\}.\]
Then $H\neq \emptyset$. If it is finite, thus for each $\{m,j\}\in H$, by  $U_{(m,j)}\neq X$, taking $x_{(m,j)}\in X/U_{(m,j)}$. Then $\{x_{(m,j)}\colon \{m,j\}\in H\}\cup K \in \kx$. Yet there isn't any element   $\{x_{(m,j)}\colon \{m,j\}\in H\}\cup K$ in $\bigcup_{m\geq n_{0}}{{\U}^{'}}_{m}$, which is a contradiction. So $H$ is infinite, then there exists a $m\geq n_{0}$, $j\leq i(m)$ such that $K\subset U_{(m,j)}$ and $f_{(m,j)}(K)\subset (-\frac{1}{m}, \frac{1}{m})$ with $\frac{1}{m}<\epsilon$. Thus $f_{(m,j)}(K)\subset (-\epsilon,\epsilon)$ and $f_{(m,j)}\in W(f_{0},K,\epsilon)$. Hence $f_{0}\in \overline{\{f_{(m,j)}\colon j\leq i(m)\}}$.
\end{enumerate}
\end{proof}

\begin{theorem}
For any Hausdorff space $X$, then the following are equivalent:
\begin{enumerate}
 \item The space $\qcx$ has countable strong fan tightness.
\item  For each  $\{{\U}_{n}\}_{n\in\N}$ sequence of open $k$-covers of $X$, there exists a ${U}_{n}\in \{\U\_{n}\}$ such that $\{U_{n}\}_{n\in \N}$ is an open $k$-cover of $X$.
\end{enumerate}
\end{theorem}
\begin{proof}\hfill
\begin{enumerate}
\item $(1)\implies(2)$. Let $\{{\U}_{n}\}_{n\in\N}$ be a sequence of open $k$-covers of $X$. Define
\[A_{n}=\{f \in \qcx \colon \,\text{For some} \, U \in {\U}_{n}\, \text{s.t.}\, f(X/U)\subset \{0\}, \forall n \in \N \}.\]
Since ${\U}_{n}$ is an open $k$-cover for $X$, then for any open subset $W(f, K, \epsilon)$ of $\qcx$, there exists a set $U$ in ${\U}_{n}$ such that $K\subset U$. Take $g \in W(f,K,\epsilon)$, then we have $h\in \qcx$ defined as
\[h(x) = g(x)\, \text{if} \, x\in U\, \text{and} \, h(X/U)\subset \{0\}.\]
Then $h\in A_{n}\cap W(f, K, \epsilon)$, therefore, $A_{n}$ is dense in $\qcx$.
Now, let's choose $g_{1}\in \qcx$ defined as $g_{1}(x)= 1, \,\forall\,x\in X$. Since each $A_{n}$ is dense, thus $g_{1}\in \bigcap_{n = 1}^{\infty}\overline{A_{n}}$. Then there is a function $f_{n}\in A_{n}$ such that $g_{1} \in \overline{\{f_{n}\colon n\in \N\}}$. So, for every $n\in\N$ and function $f_{n}\in A_{n}$, there is a set $U_{n}\in {\U}_{n}$ such that  $f_{n}(X/U)\subset \{3\}$. Now take the collection of  all such $U_{n}$'s  denoted as  $\{U_{n}:n\in\N\}$. Next, we prove $\{U_{n}:n\in\N\}$ is a $k$-cover of $X$. For any $K\in\kx$, since $g_{1} \in \overline{\{f_{n}\colon n\in \N\}}$, so there exists $f_{n}\in W(g_{1},K,1)$, for some $n\in\N$. Then, there is  a set $U_{n}$  with  $K\subset U_{n}$. Therefore, $\{U_{n}:n\in\N\}$ is an open $k$-cover of $X$.

\item $(2)\implies(1)$. Let $f_{0}$ be the zero function and $f_{0}\in \bigcap_{n\in \N} \overline{A_{n}}$, where $A_{n}\subset \qcx$. For each $n\in \N$ and $f\in A_{n}$, there exists some open subset of $X$ such that the image of that set under function $f$ is contained in $(-\frac{1}{n},\frac{1}{n})$, we denote it as $O_{n,f}$. Construct
\[{\U}_{n}=\{O_{n,f}\colon f\in A_{n}\}.\]
For each $K\in \kx$, there exists an $f\in W(f_{0},K,1/n)\cap A_{n}$ such that $K\subset O_{n,f}$. Thus, ${\U}_{n}$ is an open $k$-cover of $X$. Let $M=\{n\in \N \colon X\in {\U}_{n}\}$.
\begin{description}
  \item[Case-1] If $M$ is infinite. For every $W(f_{0},K,\epsilon)$ neighborhood of $f_{0}$ and $\epsilon>0$, there exists some $m\in\N$ such that $\frac{1}{m}<\epsilon$. Then, from the construction of ${\U}_{m}$, we have some $g_{m}\in A_{m}$ with $g_{m}(X)=(-\frac{1}{m},\frac{1}{m})$ and $g_{m}\in W(f_{0},K,\epsilon)$. Thus, the sequence $\{g_{m}\}_{m\in \N}$ is converges to $f_{0}$.
  \item[Case-2] If $M$ is finite. Then  there exists $n_{0}\in\N$ such that \[g(X)\neq (-\frac{1}{m},\frac{1}{m})\, \, \text{whenever}\, m\geq n_{0},\,g\in A_{m}.\]
Since $\{{\U}_{m}\}_{m\geq n_{0}}$ is a sequence of open $k$-cover, we have  a finite subset ${{\U}'_{m}}$ of ${\U}_{m}$ such that $\bigcup_{m\geq n_{0}}{{\U}'}_{m}$ is an open $k$-cover of $X$.
For $m\geq n_{0}$, there exists $f_{m}\in A_{m}$ such that $f_{m}(U_{m})=(-\frac{1}{m},\frac{1}{m})$. Next, we claim that $f_{0}\in \overline{\{f_{m}\colon m\geq n_{0}\}}$. For any neighborhood $W(f_{0},K,\epsilon)$ of $f_{0}$, let ${\U}_{K}=\{U_{m}\colon K\subset U_{m}, m\geq n_{0}\}$, clearly ${U}_{K}\neq \emptyset$. Let us assume ${\U}_{K}$ is finite, then  ${\U}_K=\{U_{m_{j}}\colon j\leq p\}$, where $p$ is some finite natural number. By $U_{m_{j}}\neq X$, we take $x_{m_{j}}\in X/U_{m_{j}}$, then $\{X_{m_{j}}\colon j\leq p\}\cup K$. So, $U_{m}\cap (\{x_{m_j}\colon j\leq p\}\cup K)=\emptyset$ whenever $m\neq n_{0}$. This is a contradiction, therefore, ${\U}_{K}$ is infinite. Hence there exists a natural number $m\neq n_{0}$ satisfying $K\subset U_{m}$ with $f_{m}(K)\subset f_{m}(U_{m})\subset (-\frac{1}{m},\frac{1}{m})$, with  $\frac{1}{m}< \epsilon$, then $f_{m}\in W(f_{0},K,\epsilon)$. This implies that $f_{0}\in  \overline{\{f_{m}\colon m\geq n_{0}\}}$.
\end{description}
\end{enumerate}
\end{proof}

\begin{lemma}\label{l:quasicontinuous function separating finite no. of  closed set}
Let $X$ be a regular space and $Y$ be any space. Then, for a finite number of nonempty disjoint closed subsets $F_{1}, F_{2}, \dots, F_{n}$ of $X$ and $y_{1}, y_{2}, \dots, y_{n}\in Y$, there exists a quasicontinuous function $f\colon X \to Y$ satisfying $f(F_{i})=\{y_{i}\}$ for all $1\leq i\leq n$.
\end{lemma}
\begin{proof}
Since $F_{1}, F_{2},\dots, F_{n}$ are compact subsets of $X$. For  each $i$, $1\leq i\leq n$, there exists a $f_{i}\in\qcx$ defined as
\[f_{i}(x)=
\begin{cases}
1 & \text{if $x\in{F_{i}},$}\\
0 & \text{otherwise}.
\end{cases}\]
Then we can define a quasicontinuous function $f\colon X\to Y$ as
\[f(x) =
\begin{cases}
y_{i}f_{i}(x)  & \text{if $x\in{F_{i}},$}\\
0 & \text{otherwise}.
\end{cases}\]
It satisfies $f(F_{i})=\{y_{i}\}$ for all $1\leq i \leq n$.
\end{proof}

\begin{theorem}
Let $X$ be a regular space. Then $d(\qcx)\leq kcof(X)$.
\end{theorem}
\begin{proof}
Let $kcof(X)=|\V|$, where $\V$ is a cofinal family in $\kx$. Now assume $\mathcal{M}$ to be a collection of all finite pairwise disjoint sets of the cofinal family $\V$. Let $K=\{K_{1}, K_{2},\dots, K_{n}\}\in \mathcal{M}$ and $r=\{r_{1}, r_{2},\dots, r_{n}\}$ be a finite set of rational numbers. Then, by Lemma \eqref{l:quasicontinuous function separating finite no. of  closed set}, there is a quasicontinuous function such that
$g_{K,r}(K_i)=\{r_{i}\}$ for all $1\leq i\leq n$. Consider a family
\[D=\{g_{K,r}\colon K\in \mathcal{M}\, and\, r\in \mathbb{Q} \}\]
It is clear that $|D|\leq |\V|\leq \eta$. Now, it is sufficient to prove $D$ is dense in $\qcx$. For any $f\in \qcx$,  $r\in \mathbb{Q}$ and  $C\in \kx$.  Let us define a quasicontinuous function
\[g_{K',r}(x) =
\begin{cases}
f(x)  & \text{if $x\in C$},\\
r &  \text{if $x\in X/C$}.
\end{cases}\]
Where $K'=\{C'\}$ and set $C'\in \mathcal{M}$ with $C'\subset X/C$.
Therefore, any neighborhood $W(f,C,\epsilon)$ of $f$,  there exists a  quasicontinuous function $g_{K',r}$ lies in
 $W(f,C,\epsilon)$. Hence, the set $D$ is dense.
\end{proof}
\section{Frechet-Urysohn properties of the space $\qcx$}\label{s:Frechet-Urysohn Property of qcx}
A space has  Fréchet-Urysohn property if, for any closed set and any point not in that set, there exists a sequence of points from the set that converges to the point outside the set. This property is important for interpreting a space's structure in terms of sequences because it says that points in the space can be ``approximated" by sequences from closed sets. The Fréchet-Urysohn property illustrates the relation between sequences and set closure and seems to be focused on ensuring that closed sets are ``sequentially dense" in some manner. This makes it an important aspect of the topological study of convergence and sequentially. Furthermore, we examine the behavior of different types of covers for $X$ within $\qcx$, when ever $\qcx$ is taken as a Fréchet-Urysohn (or Whyburn) space.

\begin{theorem}\label{t:Qx frechet then k cover has countable sub k cover}
If $X$ is Hausdorff space and  $\qcx$ is a Frechet-Uryshon space, then each open $k$-cover of $X$ has a countable $k$-subcover of $X$.
\end{theorem}
\begin{proof}
Suppose $\qcx$ be a Frechet-Urysohn space and $\U$ be any open $k$-cover of $X$. Thus, for each $K\in\kx$ there is a set $U_{K}\in \U$ such that $K\subset U_{K}$. Then $f_{K}\in \qcx$ such that $f_{K}(\overline{U_{K}})=\{0\}$ and $f_{K}(X/\overline{U_{K}})=\{1\}$. Take the zero function $f_{0}\in \qcx$ and let $W(f_{0},K,\epsilon)$ be any neighborhood of $f_{0}$. For $K$, there exists $U_{K}$ such that $K\subset U_{K}$. Thus, we have  a function $f_{K}$ that lies in $W(f_{0},K,\epsilon)$. Therefore, $f_{0}\in \overline{\{f_{K}\colon K\in \kx\}}$. Then by definition of Frechet-Urysohn space, there exists a sequence $\{f_{K_{n}}\colon n\in\N\}$ converging to $f_{0}$. Construct  the collection $\{U_{K_{n}}\colon n\in \N\}$, which is clearly countable. It remains to prove that it is $k$-cover of $X$. Let any $K\in \kx$, there exists some $n\in \N$ such that   $f_{K_{n}}\in W(f_{0},K,1)$, then $K\subset U_{K_{n}}$. Therefore,  $\{U_{K_{n}}\colon n\in \N\}$ is a $k$-cover of $X$.
\end{proof}

In 2007,  Ferrando and  Moll proved that  [\cite{Ferrando2007}, Theorem 1], for a locally compact Hausdorff space $X$  and  $\cx$ endowed with the compact-open topology, the following conditions are equivalents: (a) $\cx$ is a Frechet-Urysohn space, (b) the tightness of $\cx$ is countable, and (c) $X$ is $\sigma$-compact. Now, we provided the following results for the space $\qcx$.
\begin{theorem}\label{t:frechet equivalent condition}
For a locally compact Hausdorff space $X$ then the following  are equivalent;
\begin{enumerate}
\item The tightness of $\qcx$ is countable.
\item The space $\qcx$ be Frechet-Urysohn.
\item The space $X$ is $\sigma$-compact.
\end{enumerate}
\end{theorem}
\begin{proof}\hfill
\begin{enumerate}
\item $(\textit{2})\Rightarrow (\textit{1})$. Clearly, by Theorem \eqref{t:Qx frechet then k cover has countable sub k cover} and Theorem \eqref{t:Compact lindelof no. of x and tightnessof qcx}.
\item $(\textit{1})\Rightarrow (\textit{3})$. Let $t(\qcx)=\alephnull$ and $\U$ be any open $k$-cover of $X$.  Then, by Theorem \eqref{t:Tightness and k cover}, there exists a $k$-subcover ${\U}'$ of $X$ with $|{\U}'|\leq \alephnull$. Take $\mathcal{M}=\{\overline{U}\colon U\in {\U}'\}$. Since $X$ is locally compact, each set in  $\mathcal{M}$ is compact. Now it is sufficient to prove that $\mathcal{M}$ is a cofinal subfamily in $\kx$. For every $K\in\kx$, there exists some set $\overline{U}\in \mathcal{M}$ that contains $K$. Therefore, $\mathcal{M}$ is a cofinal subfamily in $\kx$. Hence, $X$ is hemicompact.
\item $(\textit{3})\Rightarrow (\textit{2})$. Let $X$ is $\sigma$-compact. Then, definition  of $\sigma$-compactness, $X=\bigcup_{n\in\N}K_{n}$, where $K_{n}$ is a compact subset of $X$ for all $n\in\N$. Thud, for every $K\in\kx$ there exists a $K_{j}$ for  some $j\in \N$ such that $K\subset K_{j}$. Construct  $K_{i}'=\bigcup_{n=1}^{i}K_{n}, \, \forall \, i\in \N$. Since the space $\qcx$ is homogenous. Let $f_{0}$ be the zero function and $F$ be any subset of $\qcx$ such that $f_{0}\in \overline{F}$. Then for  every $W(f_{0},K_{i}',\frac{1}{n})$ neighborhood of $f_{0}$, there exists some function $f_{K_{i}',i}\in W(f_{0},K_{i}^{'},\frac{1}{n})\cap F$ for all $i\in \N$. Now assume that $g_{i}=f_{K_{i}',i}$ for all $i
\in \N$. Next, we claim that sequence $\{g_{i}\colon i\in \N\}$ is converges to $f_{0}$. Let any $K\in \kx$, $\epsilon>0$, there exists some $ p\in \N$ with $\frac{1}{p}<\epsilon$, also, there is $K_{m}$ for some $m\in \N$ such that $K\subset K_{m}$, which implies $K\subset K_{i}'$ for all $i\geq  m$.
\begin{description}
\item[Case-1] If $m\geq p$. Then $\frac{1}{m}\leq \frac{1}{p}$. Thus,
\[|g_{i}(x)-f_{0}(x)|<\epsilon, \, \forall x \in K, \, \forall\, i\geq m\]
\item[Case-2] If $p>m$. Since $K\subset K_{i}'$ for all $i\geq m$. Thus, $K\subset K_{i}'$ for all $i\geq p$. Then, we have
\[|g_{i}(x)-f_{0}(x)|<\epsilon,\, \forall x \in K,\, \forall\, i\geq p\]
Therefore, the sequence $\{g_{i}\colon n\in N\}$ converges to $f_{0}$. Hence, $\qcx$ is a Frechet-Urysohn space.
\end{description}
\end{enumerate}
\end{proof}

\begin{corollary}
If $X$ is a locally compact metric space, then the following statements are equivalent:
\begin{enumerate}
\item The space $\qcx$ is a Frechet-Urysohn space.
\item The space $\qcx$ has countable tightness.
\item The space $X$ is separable.
\end{enumerate}
\end{corollary}
\begin{proof}\hfill
\begin{enumerate}
\item $(\textit{2})\Rightarrow (\textit{3})$. The space  $\qcx$ has countable tightness. Then, by Theorem \eqref{t:frechet equivalent condition}, it is $\sigma$-compact. Therefore, $\qcx$ is separable.
\item $(\textit{3})\Rightarrow (\textit{1})$. Since any separable and locally compact metric space $X$, is $\sigma$-compact. Thus, by Theorem \eqref{t:frechet equivalent condition}, $\qcx$ is Frechet-Urysohn space.
\item $(\textit{1})\Rightarrow (\textit{2})$. From, by Theorem \eqref{t:Qx frechet then k cover has countable sub k cover} and Theorem \eqref{t:Compact lindelof no. of x and tightnessof qcx}.
\end{enumerate}
\end{proof}

\begin{definition}
A family $\mathcal{U}$ of subsets of a space $X$ is called a $k_{f}$-cover if for any finite subfamily ${\mathcal{K}}'$ of $\kx$ there exists $O\in\mathcal{U}$ such that $\bigcup {\mathcal{K}}' \subset O$. If each element of $\mathcal{U}$ is open, it is known as an open $k_{f}$-cover.
\end{definition}
Suppose $X$ is any space and $\gamma$ is a family of subsets of $X$. Let $\lim \gamma=\{z\in X\colon |\{U\in \gamma\colon z\notin U\}<\alephnull|\}$, and if $g$ be real-valued function on $X$, then $supp(g)=\{z\in X\colon g(z)\neq 0\}$.

\begin{theorem}
Let $X$ be any Hausdorff space and $\qcx$ be a Frechet-Urysohn space. Then for any open $k_{f}$-cover $\U$ of  $X$, there exists a countable subfamily $\xi$ of  $\U$ such that $\lim \xi= X$.
\end{theorem}
\begin{proof}
Let $\qcx$ be a Frechet-Uryshon space. If $X\in \U$, then $\xi=X$, but if not,then we assume that
\[M=\{f\in\qcx \colon supp(f)\subset U \,\text{for some}\, U\in\U\}\]
Let $g_{1}$  be a constant function such that $g_{1}(x) = 1$ for all $x \in X$. Then $supp(g_{1})=X \notin \U$, which implies $g_{1}\notin M$. Now, take $K_{1}, K_{2}, \dots K_{p} \in \kx$, where $K_{i}\cap K_{j}= \emptyset$, for all $ 1\leq i,j\leq p$. By the definition of $k_{f}$-cover, there is a $U\in \U$ such that $\{K_{1}, K_{2}, \dots K_{p}\}\subset U$. So  $X/U$ and $\bigcup_{i=1}^{p} K_{i}$ are two disjoint closed sets. Thus, by Lemma \eqref{l:quasicontinuous function separating point and closed set} there exists a quasicontinuous function such that
\[f(X/U)=\{0\} \,\text{and}\, f(\bigcup_{i=1}^{p} K_{i})=\{1\} \]
Thus, $f\in W(f_{1},K_{i},\epsilon)$ and $supp(f)=\bigcup_{i=1}^{p} K_{i}\subset U\in \U$. Then, we have $f\in W(f_{1},K_{i},\epsilon)\cap M$, which implies that  $f\in \overline{M}$. Therefore, $f\in\overline{M}/M$.

By the definition of a Frechet-Urysohn space $\qcx$, there exists a sequence $\{f_{n} \colon n \in \N\}\subset  M$ such that $f_{n}\to g_{1}$. By the selection of sequence, there exists some $U_{n}\in \U$ such that $supp(f_{n})\subset U_{n}$ for each $n\in \N$. Now, we claim that if $\xi = \{U_{n}\colon n\in \N\}$ then $\lim\, \xi=X$. For any  $x\in X$, then $W(g_{1},\{x\},\epsilon)$ be a  neighborhood of $f_{1}$, where $0< \epsilon < 1$. Thus, for $W(g_{1},\{x\},\epsilon)$ there is some number $m\in \N$  for that $f_{n}\in W(g_{1},\{x\},\epsilon)$ for all $n\geq m$, which implies $x\in supp(f_{n})\subset U_{n}$ for all $n\geq m$. Therefore, $\lim\, \xi=X$.
\end{proof}

\begin{theorem}
If $X$ is a Hausdorff space and $\qcx$ is a Whyburn space. Let $\{{\gamma}_{n}\}_ {n\in \N}$ be a sequence of open covers of $X$ having the following properties:
\begin{enumerate}
\item ${\gamma}_{n}=\{ U_{m}^{n}\colon m\in \N\}$ and  $U_{m}^{n}\subset  U_{m+1}^{n}$, for each $n\in \N$.
\item For any $n\in \N$, we have a closed cover ${\mathcal{U}}_{n}=\{ F_{m}^{n}\colon m\in \N\}$ of $X$ with $F_{m}^{n}\subset  U_{m+1}^{n}$ and $F_{m}^{n}\subset  F_{m+1}^{n}\, \forall\, m\in \N$.
\end{enumerate}
Then there exists a $k$-cover $\{W_{n}\colon n\in \N\}$ of $X$, where $W_{n}\in{\gamma}_{n}, \,\forall\, m\in \N$.
\end{theorem}
\begin{proof}
Let $(m,n)$ be any pair of natural numbers. Then there exists a $f_{m}^{n}\in \qcx$ such that $f_{m}^{n}|F_{m}^{n}\equiv \frac{1}{n}$ and $f_{m}^{n}|(X/U_{m}^{n})\equiv 1$. Now, take a sequence $S_{n}=\{f_{m}^{n}\colon m \in \N\}$, which converges to $h_{n}\equiv \frac{1}{n}$. Set $S =\bigcup_{n\in\N}S_{n}$ and a zero function $h$, then clearly $h$ lies in $\overline{S}$. Since $\qcx$ is a whyburn space, thus  we have  a subset $F$ of $S$ with  $\overline{F}=F\cup\{h\}$. Therefore, for any $n\in\N$, the set $F=F\cap S_{n}$ will not contain because $\overline{F}/F=\{h\}$ otherwise $h_{n}\in \overline{F}/F$. Therefore, for every $n\in\N$, we have a number $m(n)\in \N$ such that $F_{n}\subset \{f_{m}^{n}\colon m \in \N\}$. So for each $n\in\N$ take $W_{n}=\{U_{m(n)}^{n}\}$. Next, to prove $\{W_{n}\colon n\in\N\}$ is $k$-cover of $X$. Since $h\in\overline{F}$, then for each $K\in \kx$, there exists a $f_{m}^{n}\in F$ such that $f_{m}^{n}(x)<1$ for all $x\in K$. Therefore, $K\cap(X/U_{m}^{n})=\emptyset$, thus $K\subset U_{m}^{n}\subset U_{m(n)}^{n}$. Hence, $\{W_{n}\colon n\in\N\}$ is a $k$-cover of $X$.
\end{proof}
\section{Conclusion}
In this paper, we analyzed the density and various forms of tightness. We proved that while the density of $\qcx$ is smaller than the $k$-cofinality of $X$, the tightness of $\qcx$ is the same as the compact Lindelöf number of the Hausdorff space $X$. We additionally obtained requirements on $X$ such that the tightness of $\qcx$ coincides with the density-tightness. We also used countable cover kinds of $X$ to characterize fan tightness and high fan tightness for $\qcx$. Furthermore, it was demonstrated that the features of $\qcx$ being Frechet-Urysohn, possessing countable tightness, and $\sigma$-compactness of $X$ are identical for a locally compact Hausdorff space $X$. Finally, we demonstrate that if $\qcx$  is a Frechet-Urysohn space, then every $k_{f}$-open covering of $X$ has a countable subcover that converges to $X$.


\end{document}